\documentclass{asl}
\usepackage{color}

\newtheoremstyle{mystyle}{}{}{\slshape}{2pt}{\scshape}{.}{ }{} 

\newtheorem{thm}{Theorem}
\newtheorem{cor}[thm]{Corollary}

\newtheorem{prop}[thm]{Proposition}
\newtheorem{lemme}[thm]{Lemma}
\newtheorem{fait}[thm]{Fact}

\theoremstyle{definition}

\theoremstyle{mystyle}

\theoremstyle{remark}

\newcommand{\impl}{\rightarrow}
\newcommand{\pprec}{\prec^+}
\newcommand{\monster}{\mathcal U}

\DeclareMathOperator{\tp}{tp}
\DeclareMathOperator{\dcl}{dcl}
\DeclareMathOperator{\acl}{acl}
\DeclareMathOperator{\alt}{alt}

\title{On forking and definability of types in some dp-minimal theories}
\author{Pierre Simon}
\revauthor{Simon, Pierre}
\address{Universit\'e de Lyon; CNRS\\
Universit\'e Lyon 1\\
Institut Camille Jordan UMR5208\\
43 boulevard du 11 novembre 1918\\
69622 Villeurbanne Cedex, France}

\thanks{P. S. was partially supported by the European Research Council under the European Unions Seventh Framework Programme (FP7/2007-2013) / ERC Grant agreement no. 291111 and by ValCoMo (ANR-13-BS01-0006).}

\author{Sergei Starchenko}
\revauthor{Starchenko, Sergei}
\address{Department of Mathematics \\
University of Notre Dame\\
Notre Dame, IN 46556, USA}
\thanks{S. S. was partially supported by NSF.}

\begin{document}
\begin{abstract}We prove in particular that, in a large class of dp-minimal theories including the p-adics, definable types are dense amongst non-forking types.
\end{abstract}
\maketitle

\section{Introduction and preliminaries}
In this short note, we show how the techniques from \cite{InvTypes} can be adapted to prove the \emph{density} of definable types in a large class of dp-minimal theories. Density of definable types is the following: for any $\phi(x)$ which does not fork over a model $M$, there is a global type $p(x)$ definable over $M$ and containing $\phi(x)$. We prove this for dp-minimal $T$ satisfying an extra property---property (D)---which says that unary definable sets contain a type that is definable over the same parameters as the set. This holds in particular if definable sets have natural \emph{generic} definable types. This also holds whenever $T$ has definable Skolem functions. In particular our theorem applies to the field $\mathbb Q_p$ of $p$-adic numbers.

\smallskip
Throughout, $T$ is a complete countable theory. We let $\monster$ be a monster model. By a global type, we mean a type over $\monster$.  We write $M \pprec N$ to mean $M\prec N$ and $N$ is $|M|^+$-saturated.

The notation $\phi^0$ means $\neg \phi$ and $\phi^1$ means $\phi$.

If $M\pprec N$ and $p\in S(N)$, then $p$ is $M$-invariant if for any $b,b'\in N$ and any formula $\phi(x;y)$, $b\equiv_M b'$ implies $p\vdash \phi(x;b) \leftrightarrow \phi(x;b')$. Any $M$-invariant type over $N$ extends in a unique way to a global $M$-invariant type. Thus there is no harm in considering only global invariant types.

\smallskip
We refer to \cite{InvTypes} or to \cite{NIP_book} for basic facts about NIP theories, though we will now collect all the statements that we need.

First recall that in an NIP theory, a global type $p$ does not fork over a model $M$ if and only if it is $M$-invariant.

If $p(x)$ and $q(y)$ are two global $M$-invariant types, then $p(x)\otimes q(y)$ denotes the global type $r(x,y)$ defined as $\tp(a,b/\monster)$ where $b\models q$ and $a\models p|\monster a$ (invariant extension of $p$ to $\monster a$).

If $p(x)\otimes q(y)=q(y)\otimes p(x)$, then we say that $p$ and $q$ \emph{commute}. It is not hard to see that, in any theory $T$, a global $M$-invariant type is definable if and only if it commutes with all global types finitely satisfiable in $M$ (see \cite[Lemma 2.3]{InvTypes}).

Next we recall the notion of strict non-forking from \cite{CherKapl}. Let $M$ be a model of an NIP theory. A sequence $(b_i)_{i<\omega}$ is strictly non-forking over $M$ if for each $i<\omega$, $\tp(b_i/b_{<i}M)$ is strictly non-forking over $M$ which means that it extends to a global type $\tp(b_*/\monster)$ such that both $\tp(b_*/\monster)$ and $\tp(\monster/Mb_*)$ are non-forking over $M$. We will only need to know two facts about strict non-forking sequences (both proved in \cite{CherKapl}, see also \cite[Chapter 5]{NIP_book}):

\smallskip
(Existence) Given $b\in \monster$ and $M\models T$, there is an indiscernible sequence $b=b_0,b_1,\ldots$ which is strictly non-forking over $M$. We call such a sequence a \emph{strict Morley sequence} of $\tp(b/M)$.

\smallskip
(Witnessing property) If the formula $\phi(x;b)$ forks over $M$, then for any strictly non-forking indiscernible sequence $b=b_0,b_1,\ldots$, the type $\{\phi(x;b_i):i<\omega\}$ is inconsistent.

\smallskip
If $\phi(x;y)$ is an NIP formula, we let $\alt(\phi)$ be the \emph{alternation number} of $\phi$, namely the maximal $n$ for which there is an indiscernible sequence $(b_i:i<\omega)$ and a tuple $a$ with $\neg (\phi(a;b_i)\leftrightarrow \phi(a;b_{i+1}))$ for all $i<n$. If $(b_i:i<\omega)$ is indiscernible and $\{\phi(x;b_i):i<\alt(\phi)/2+1\}$ is consistent, then $\{\phi(x;b_i):i<\omega\}$ is also consistent.

\smallskip

We will also need the notion of ``$b$-forking" as defined in Cotter and Starchenko's paper \cite{CotStar} and as recalled in \cite{InvTypes}. For this, we assume that $T$ is NIP.

Assume we have $M\prec^{+} N$ and $b\in \monster$ such that $\tp(b/N)$ is $M$-invariant. We say that a formula $\psi(x,b;d)\in L(Nb)$ $b$-divides over $M$ if there is an $M$-indiscernible sequence $(d_i:i<\omega)$ inside $N$ with $d_0=d$ and $\{\psi(x,b;d_i):i<\omega\}$ is inconsistent. We define $b$-forking in the natural way.

\begin{fait}($T$ is NIP)
Notations being as above, the following are equivalent:

(i) $\psi(x,b;d)$ does not $b$-divide over $M$;

(ii) $\psi(x,b;d)$ does not $b$-fork over $M$;

(iii) if $(d_i:i<\omega)$ is a strict Morley sequence of $\tp(d/M)$ inside $N$, then $\{\psi(x,b;d_i):i<\omega\}$ is consistent;

(iii)$'$ if $(d_i:i<\omega)$ is a strict Morley sequence of $\tp(d/M)$ inside $N$, then $\{\psi(x,b;d_i):i<m\}$ is consistent where $m$ is greater than the alternation number of $\psi(x,y;z)$;

(iv) there is $a\models \psi(x,b;d)$ such that $\tp(a,b/N)$ is $M$-invariant.
\end{fait}

\medskip
Finally a theory $T$ is dp-minimal if for every $A\subset \monster$, every singleton $a$ and any two infinite sequences $I_0,I_1$ of tuples, if $I_k$ is indiscernible over $AI_{1-k}$, $k=0,1$, then for some $k\in\{0,1\}$, $I_k$ is indiscernible over $Aa$.

Any o-minimal or weakly o-minimal theory is dp-minimal, as is the theory of the fields of p-adics.

The following theorem was proved in \cite{InvTypes}:

\begin{thm}\label{dim-one}
($T$ is dp-minimal) Let $p(x)$ be a global $M$-invariant type in a single variable, then $p$ is either definable over $M$ or finitely satisfiable in $M$.
\end{thm}

\section{The main theorem}

We will say that $T$ has property (D) if for every set $A$ (of real elements) and consistent formula $\phi(x)\in L(A)$, with $x$ a single variable, there is an $A$-definable complete type $p\in S_{x}(A)$ extending $\phi(x)$.

We emphasise that the type $p$ might not extend to a global $A$-definable type.

\begin{lemme}\label{lem_Dplus}
Let $M\prec N$ and $b\in \monster$ such that $\tp(b/N)$ is $M$-definable. Assume that $p\in S_{x}(Mb)$ is a complete $Mb$-definable type, then $p$ extends to a complete type $q\in S_x(Nb)$ which is $Mb$-definable using the same definition scheme as $p$.
\end{lemme}
\begin{proof}
For each formula $\phi(x;y,b)\in L(b)$, there is by hypothesis a formula $d\phi(y;b)\in L(M)$ such that for every $d\in M^{|y|}$ we have $p\vdash \phi(x;d,b)$ if and only if $\monster \models d\phi(d;b)$. We have to check that the scheme $\phi(x;y,b) \mapsto d\phi(y;b)$ defines a consistent complete type over $Nb$. This follows at once from the fact that $\tp(b/N)$ is an heir of $\tp(b/M)$. Let us check completeness for example. Assume that there is some $n\in N$ and formula $\phi(x;y,b)$ such that $\monster \models \neg d\phi(n;b)\wedge \neg d(\phi^0)(n;b)$. By the heir property, there must be such a tuple $n$ in $M$, which is a contradiction.
\end{proof}

\begin{lemme}\label{lem_induct}($T$ is NIP)
Let $M\pprec N$, $n<\omega$ and assume that any formula $\theta(y;d)\in L(N)$ with $|y|=n$ and non-forking over $M$ extends to an $M$-definable type over $N$. Let $\phi(x,y;d)\in L(N)$ be non-forking over $M$, where $|y|=n$ and $|x|=1$. Then we can find a tuple $(a,b)\models \phi(x,y;d)$ such that  $\tp(a,b/N)$ is $M$-invariant and $\tp(b/N)$ is definable (over $M$).
\end{lemme}
\begin{proof}
Let $(d_i:i<\omega)$ be a strict Morley sequence of $\tp(d/M)$ inside $N$. Let $m<\omega$ be greater than the alternation number of $\phi(x,y;z)$. As the formula $\phi(x,y;d)$ does not fork over $M$, it extends to a global $M$-invariant type $p$. Then the conjunction $\psi(x,y;\bar d)=\bigwedge_{i<m} \phi(x,y;d_i)$ is in $p$. In particular it is consistent and does not fork over $M$. The same is true for $\theta(y;\bar d)=(\exists x)\psi(x,y;\bar d)$. By hypothesis, we can find some $b\in \monster$ such that $\tp(b/N)$ is $M$-definable and $\monster \models \theta(b;\bar d)$. We claim that the formula $\phi(x;b,d)$ does not $b$-fork over $M$. Assume that it did. Then the conjunction $\bigwedge_{i<m} \phi(x,b;d_i)$ would be inconsistent. But this contradicts the fact that $\theta(b;\bar d)$ holds. Hence we may find $a\in \monster$ such that $\phi(a,b;d)$ holds and $\tp(a,b/N)$ does not fork over $M$ (equivalently is $M$-invariant).
\end{proof}

\begin{thm}\label{th_propD}
Assume that $T$ is dp-minimal and has property (D). Let $M\models T$ and $\phi(x;d)\in L(\monster)$ be non-forking over $M$. Then $\phi(x;d)$ extends to a complete $M$-definable type.
\end{thm}
\begin{proof}
The proof is an adaptation of the argument given for Proposition 2.7 in \cite{InvTypes}. We argue by induction on the length of the variable $x$.

\smallskip\noindent
\underline{$|x|=1$}: Assume that $|x|=1$ and take $p(x)$ a global type extending $\phi(x;d)$ and non-forking over $M$. If $p$ is definable, we are done. Otherwise, by Theorem \ref{dim-one}, $p$ is finitely satisfiable in $M$. This implies that $\phi(x;d)$ has a solution $a$ in $M$. Then $\tp(a/\monster)$ does the job.

\smallskip \noindent
\underline{Induction}:
Assume we know the result for $|x|=n$, and consider a non-forking formula $\phi(x_1,x_2;d)$, where $|x_2|=n$ and $|x_1|=1$. Let $N\succ M$ sufficiently saturated, with $d\in N$. Using the induction hypothesis and Lemma \ref{lem_induct}, we can find a tuple $(a_1,a_2)\models \phi(x_1,x_2;d)$ such that  $\tp(a_1,a_2/N)$ is $M$-invariant and $\tp(a_2/N)$ is definable (over $M$).

If $p=\tp(a_1,a_2/N)$ is definable we are done. Otherwise, there is some type $q\in S(N)$ finitely satisfiable in $M$ such that $p$ does not commute with $q$.

Now let $c\in \monster$ such that $(a_1\hat{~}a_2,c)\models p\otimes q$. Let $I$ be a Morley sequence of $q$ over everything. As $\tp(a_2/N)$ is definable, it commutes with $q$. Therefore the sequence $\bar c=c+I$ is indiscernible over $Na_2$. However, it is not indiscernible over $Na_1a_2$. Take some $M\pprec N_1\pprec N$ with $\tp(N_1/Md)$ finitely satisfiable in $M$. 

Take $r\in S(\monster)$ finitely satisfiable in $N$. Let $b\models r|_{Na_2\bar c}$. Build a Morley sequence $J$ of $r$ over everything. Then $b+J$ is indiscernible over $Na_2\bar c$ and $\bar c$ is indiscernible over $NbJ$. As $\bar c$ is not indiscernible over $Na_1a_2$, by dp-minimality, $b+J$ must be indiscernible over $Na_1a_2$. Hence $b\models r|_{Na_1a_2\bar c}$.

We have shown that $r|_{Na_2\bar c} \vdash r|_{Na_1a_2\bar c}$. Let $l=l_r \in \{0,1\}$ such that $r(y) \vdash \phi^l(a_1,a_2;y)$. Then $r(y)|_{Na_2\bar c} \vdash \phi^l(a_1,a_2;y)$. By compactness, there is a formula $\theta_r(y)$ in $r(y)|_{Na_2\bar c}$ which already implies $\phi^l(a_1,a_2;y)$. Using compactness of the space of global $N$-finitely satisifiable types, we can extract from the family $(\theta_r(y))_r$ a finite subcover $\mathcal C$. Let $\theta_l(y)$ be the disjonction of the formulas in $\mathcal C$ that imply $\phi^l(a_1,a_2;y)$. Summing up, we have:

 $\monster\models \theta_l(y) \impl \phi^l(a_1,a_2;y)$, $l=0,1$, and every type finitely satisfiable in $N$ satisfies either $\theta_1(y)$ or $\theta_2(y)$. In particular, this is true of any point $n\in N$.

Write $\theta_1(y)$ as $\theta_1(y;a_2,\bar c,e)$ exhibiting all parameters, with $e\in N$. By invariance of $\tp(a_1,a_2,\bar c/N)$, we may assume that $e\in N_1$ and in particular $\tp(e/Md)$ is finitely satisfiable in $M$.

As $\tp(\bar c/Na_2)$ is finitely satisfiable in $M$, there is $\bar c'\in M$ such that:
$$\models \theta_1(d;a_2,\bar c', e) \wedge (\exists x)(\forall y)(\theta_1(y;a_2,\bar c', e)\impl \phi(x;y)).$$

Next, $\tp(e/Md)$ is finitely satisfiable in $M$. As $\tp(a_2/N)$ is $M$-definable, also $\tp(e/Mda_2)$ is finitely satisifiable in $M$ and we may find $e'\in M$ such that the previous formula holds with $e$ replaced by $e'$.

By property (D), there is some $Ma_2$-definable type $p_1(x_1)\in S(Ma_2)$ containing the formula $(\forall y)(\theta_1(y;a_2,\bar c',e')\impl \phi(x;y))$. By Lemma \ref{lem_Dplus}, $p_1$ extends to a complete $Ma_2$-definable type over $Na_2$. Let $a'_1$ realise that type. Then $\tp(a'_1,a_2/N)$ is $M$-definable and we have $\models \phi(a'_1,a_2;d)$ as required.
\end{proof}

%

Theorem \ref{th_propD} was proved for \emph{unpackable VC-minimal theories} by Cotter and Starchenko in \cite{CotStar}. This class contains in particular o-minimal theories (for which the result was established earlier by Dolich \cite{Dol}) and C-minimal theories with infinite branching. We show now that our result generalises Cotter and Starchenko's and covers some new cases, in particular the field of $p$-adics.

\begin{lemme}\label{lem_acldef}
Let $A$ be any set of parameters and $p(x)$ be a global $\acl(A)$-definable type. Then $p|_A$ is $A$-definable.
\end{lemme}
\begin{proof}
Take $\phi(x;y)\in L$ and let $d\phi(y;a)$, $a\in \acl(A)$, be the $\phi$-definition of $p$. Then $\tp(a/A)$ is isolated by a formula $\phi(z)\in L(A)$. Define $D\phi(y)= (\exists z)\phi(z)\wedge d\phi(y;z)$. Then $D\phi(y)$ is a formula over $A$ and defines the same set on $A$ as $d\phi(y)$.
\end{proof}

\begin{prop}\label{prop_D}
The following classes of theories have property (D):

$\bullet$ theories with definable Skolem functions;

$\bullet$ dp-minimal linearly ordered theories;

$\bullet$ unpackable VC-minimal theories.
\end{prop}
\begin{proof}
Let $T$ have definable Skolem functions and take a formula $\phi(x)\in L(A)$. Then we can find $a\in \dcl(A)$ such that $\models \phi(a)$, and thus $\tp(a/A)$ is as required.
%

\smallskip
Assume now that $T$ is dp-minimal and that the language contains a binary symbol $<$ such that $T\vdash ``x<y$ defines a linear order". Let $\phi(x)\in L(A)$ be a formula with $|x|=1$. If the formula $\phi(x)$ contains a greatest element, then that element is definable from $A$, and we conclude as in the previous case. Otherwise, consider the following partial type over $\monster$: $$p_0=\{a<x : a\in \phi(\monster)\} \cup \{x<b : \phi(\monster)<b\}\cup \{\phi(x)\}.$$ Let $\mathfrak P$ be the set of completions of $p_0$ over $\monster$. By Lemma 2.8 from \cite{dpmin}, any $p\in \mathfrak P$ is definable over $M$. In particular, $\mathfrak P$ is bounded. Since $\mathfrak P$ is $A$-invariant (setwise), we conclude that every $p\in \mathfrak P$ is $\acl^{eq}(A)$-definable. Let $p$ be such a type. Then by Lemma \ref{lem_acldef} $p|_A$ is $A$-definable.
\smallskip

Finally, let $T$ be  an unpackable VC-minimal theory.  We will use results and terminology 
from \cite{CotStar}. Let  $\phi(x)\in L(A)$ be a consistent formula with $|x|=1$.
We work in $T^{\mathrm eq}$. By the uniqueness of Swiss  cheese decomposition, there is a consistent formula $\theta(x)$ over $\acl(A)$ that defines a Swiss cheese 
and $\models \theta(x)\rightarrow \phi(x)$. The outer ball $B$ of  $\theta(x)$ is definable over $\acl(A)$. The generic type the interior of $B$ (see 
\cite[Definition 2.9]{CotStar}) is a global type definable over $\acl(A)$. Now use Lemma \ref{lem_acldef}.
\end{proof}

Knowing that the theory of the $p$-adics has definable Skolem functions, we obtain the following corollary.
\begin{cor}
Let $T=Th(\mathbb Q_p)$ and $M\models T$, then any formula in $L(\monster)$ which does not fork over $M$ extends to an $M$-definable type.
\end{cor}

\bibliographystyle{asl}

\bibliography{DefTypes}

\end{document}